\documentclass[a4paper,12pt]{article}
\usepackage{amsmath, amsthm}
\usepackage{mathtools}
\usepackage{amssymb}
\usepackage{amscd}
\usepackage{epic, eepic}
\usepackage{url}
\usepackage{color}
\usepackage[T2A]{fontenc}   
\usepackage[utf8]{inputenc} 
\usepackage{comment}
\usepackage{todonotes}
\usepackage{graphicx}
\usepackage{epstopdf}
\usepackage{enumerate}
\usepackage{array}
\usepackage{tabularx}
\usepackage{tikz}
\usepackage{enumitem}
\usepackage{tcolorbox}
\usepackage{pdflscape}
\usepackage{svg}
\usepackage{hyperref}

\usepackage{caption}
\usepackage{subcaption}

\newtheorem{theorem}{\bf Theorem}[section]

\newtheorem{cor}[theorem]{\bf Corollary}
\newtheorem{proposition}[theorem]{\bf Proposition}

\newtheorem{nota}[theorem]{\bf Notation}

\newtheorem{remark}[theorem]{\bf Remark}
\newtheorem{defi}[theorem]{\bf Definition}

\newtheorem{conjecture}[theorem]{\bf Conjecture}

\title{Extending edge-colorings of distance-2 matchings in the hypercube}

\date{}

\author{Pál Bärnkopf \thanks{Alfréd Rényi Institute of Mathematics, Budapest, Hungary. Partially supported by the Counting in Sparse Graphs Lendület Research Group. This project was partially supported by the ERC grant ERMiD.
E-mail: {\tt barpal@student.elte.hu}}
}
 
\begin{document}

\maketitle

\begin{abstract}
Casselgren, Markstörm, and Pham conjectured that any precolored dis\-tan\-ce-2 matching in the $d$-dimensional cube $Q_d$ with at most $d$ colors can be extended to a proper $d$-edge-coloring. In this paper, we prove this conjecture and some related theorems. Especially, our result establishes that if $G$ is a bipartite graph, then a precolored distance-2 matching in the Cartesian product $H = G \mathbin{\Box} K_{2m}$ with at most $\chi'(H) = \Delta(H) = \Delta(G) + 2m - 1$ colors can be extended to an edge-coloring using at most $\chi'(H)$ colors. As another generalization, we establish a similar result for the Cartesian product $G \mathbin{\Box} K_{1,m}$.

\end{abstract}

\section{Introduction}

In this paper, we deal with proper edge-colorings of graphs. Throughout the paper, we often say just coloring instead of proper edge-coloring. The basic concepts not defined in the article can be found in the book \cite{diestel}. The edge chromatic number (or chromatic index) of a graph $G$ is denoted by $\chi'(G)$. According to Vizing's Theorem \cite{Vizing}, if $G$ is a simple graph with maximum degree $\Delta(G)$, then its chromatic index is either $\Delta(G)$ or $\Delta(G)+1$. For bipartite graphs, it is known that $\chi'(G) = \Delta(G)$. An (edge) \textit{precoloring} (or partial edge-coloring) of a graph $G$ is a proper edge-coloring of some edge set $E_0 \subseteq E(G)$. 

\begin{defi}
    \textup{We call a precoloring of some edge set $E_0$ \textit{extendable} if there is a proper edge-coloring with $\chi'(G)$ colors such that the colors of the edges in $E_0$ are the prescribed colors. Such a coloring is called an \textit{extension} of the precoloring.}
\end{defi}


\begin{defi}
    {\upshape The \textit{Cartesian product} $G_1 \mathbin{\Box} G_2$ of graphs $G_1=(V_1,E_1)$ and $G_2=(V_2,E_2)$ is a graph whose vertex set is the Cartesian product $V(G_1) \times V(G_2)$ and two vertices $(u_1,u_2)$ and $(v_1,v_2)$ are adjacent in $G_1 \mathbin{\Box} G_2$ if and only if either
    \begin{itemize}
        \item $u_1=v_1$ and $u_2$ is adjacent to $v_2$ in $G_2$ or
        \item $u_2=v_2$ and $u_1$ is adjacent to $v_1$ in $G_1$.
    \end{itemize}
    }
\end{defi}

\begin{nota}
    $G^d$ denotes the $d$-th power of the Cartesian product of $G$ with itself.
\end{nota}

\begin{nota}
    The $d$-dimensional hypercube, denoted $Q_d$, is the $d$-th power of the Cartesian product of $K_2$ with itself, i.e. $Q_d=K_2^d$.
\end{nota}

\begin{defi}
    \textup{In a connected graph, the \textit{distance} $d(x,y)$ between two vertices $x$ and $y$ is the length of the shortest $x-y$ path and the \textit{distance} $d(e,f)$ of two edges $e=xy$ and $f=zw$ is $\min \{d(x,z),d(x,w),d(y,z),d(y,w) \}$.}
\end{defi}

\begin{defi}
    \textup{A matching $M$ is a \textit{distance-k matching}, if for every $m_1, m_2 \in M: d(m_1,m_2) \geq k$.}
\end{defi}

Completion of partial (edge-)colorings of graphs has a long history. For instance, completion of partial Latin squares can be interpreted as an edge-coloring extension problem restricted to complete bipartite graphs, and this has been studied since as early as 1960; see, e.g. \cite{Smetaniuk}. The first known publication explicitly deals with edge-coloring extensions is by Marcotte and Seymour~\cite{Marcotte}, who studied when a particular necessary condition for extendability of a partial edge-coloring is also sufficient. Since then, it has been shown that the problem of extending a given edge-precoloring is an NP-hard problem, even for 3-regular bipartite graphs \cite{Easton, Fiala}.

The problem of extending a partial edge-coloring with distance constraints on the precolored edges goes back to the work of Albertson and Moore~\cite{Albertson}. They conjectured that if some edges of a graph $G$ are colored using at most $\Delta(G) + 1$ colors such that the distance between any two precolored edges is at least $3$, then the precoloring can be extended using at most $\Delta(G) + 1$ colors. Later, Edwards et al.~\cite{Edwards} conjectured that distance at least $2$ should be enough.

\begin{conjecture} \label{Edwards}
    Suppose that $P \subset E(G)$ is such that the distance between any pair of edges in $P$ is at least $2$. Any $(\Delta+ 1)$-edge-coloring of $P$ extends to a $(\Delta+ 1)$-edge-coloring of all of $G$.
\end{conjecture}

Recently, extending (or avoiding) an edge-precoloring in hypercubes has become an active area of research; see, e.g. \cite{cassel-hyp-1, cassel-hyp-2, cassel-hyp-3}. However, an important difference is that they do not just study the above conjecture on special classes of graphs; rather, the question is typically when a given coloring can be extended to a $\chi'(G)$-edge-coloring. In the paper~\cite{cassel-hyp-2} the authors pose the following conjecture. (An induced matching is equivalent to a distance-2 matching.)


\begin{conjecture} \label{hypercube} \cite{cassel-hyp-2}
    If $\varphi$ is an edge-precoloring of $Q_d$ using at most $d$ colors where all precolored edges lie in an induced matching, then $\varphi$ is extendable to a proper $d$‐edge-coloring.
\end{conjecture}



In \cite{cassel-hyp-3}, Casselgren, Markström and Pham established the conjecture under the stronger assumption that the distance between any two precolored edges is at least three and also showed in \cite{cassel-hyp-2} that the conjecture is valid when all precolored edges are contained in at most two distinct dimensional matchings. Furthermore, results from \cite{west} imply that Conjecture~\ref{hypercube} holds in the special case where all precolored edges receive the same color. In the paper \cite{gyori}, Győri and the author prove the following theorem.

\begin{theorem} \cite{gyori}
    Let $G$ be a bipartite graph with maximum degree $\Delta(G)$. Suppose that the color of some edges of the Cartesian product $G \mathbin{\Box} K_{2}$ is prescribed using at most $\Delta(G)+1$ colors such that the distance between any two precolored edges is at least $3$. Then, this precoloring is extendable.
\end{theorem}

In this paper, we prove a strengthening of the statement, namely that it also holds for distance-2 matchings and $K_{2m}$.

\begin{theorem} \label{main}
    Let $G$ be a bipartite graph with maximum degree $\Delta(G)$ and $m \geq 1$ an integer. Suppose that the color of some edges of the Cartesian product $H = G \mathbin{\Box} K_{2m}$ is prescribed using at most $\chi'(H) = \Delta(G) + 2m - 1$ colors such that the distance between any two precolored edges is at least $2$. Then, this precoloring is extendable.
\end{theorem}

The special case of the statement where $G = Q_{d-1}$ and $m=1$ implies Conjecture \ref{hypercube}.

The proof is based on the following result by Borodin, Kostochka, and Woodall:

\begin{theorem} \cite{Kostochka} \label{Kostochka}
     Let $G$ be a bipartite graph and set $f(e)=\max\{d(u),d(w)\}$ for each edge $e=uw$ in $G$. If we assign to each edge $e$ of $G$ a color list $L(e)$ of size at least $f(e)$, then $G$ has a proper edge-coloring in which the color of each edge $e$ belongs to its list $L(e)$.
\end{theorem}

\section{Proof of Theorem \ref{main}}

\begin{proof}[Proof of Theorem \ref{main}]
    Let $V(K_{2m})=\{a_1,\dots, a_{2m}\}$. To prove the theorem, we consider a list edge-coloring problem on a (sub)graph of $G$. Initially, assign the list of colors $\{1,2,...,\Delta(G)+2m-1\}$ to every edge of $G$. If an edge of the form $(u,a_i)(v,a_i)$ for some $i \in \{1, \dots, 2m \}$ has a prescribed color, then we remove the edge $uv$ from $G$ and delete the prescribed color from the lists of all edges adjacent to $uv$. Note that for every edge $uv$, there exists at most one $i$ such that $(u,a_i)(v,a_i)$ is precolored due to the distance condition. If there is a prescription on an edge of the form $(u,a_i)(u,a_j)$ for some $i \neq j$, then we delete the prescribed color from the list of all edges incident to $u$.
 
    The resulting graph $G'$ is bipartite and has maximum degree at most $\Delta(G)$. Moreover, we delete at most two colors from the color list of each edge, so if $m > 1$, then every color list contains at least $\Delta(G)$ colors. If the color list of an edge contains fewer than $\Delta(G)$ colors (which can only occur when $m=1$), then both of its endpoints are incident to a removed edge by the distance condition and hence each has degree at most $\Delta(G) - 1$. By Theorem~\ref{Kostochka}, $G'$ admits a proper list edge-coloring from these lists. This coloring of $E(G')$ can then be extended by incorporating the prescribed colors, yielding a coloring of $E(G)$ in all copies of $G$ within the Cartesian product. 

    Since the set of colors assigned to the edges incident to $u$ is the same in all copies of $G$, there are always $2m-1$ colors remaining to color the edges of the complete graph induced by the vertices $(u,a_i) \quad 1 \leq i \leq 2m$. This suffices, as it is known that $\chi'(K_{2m})=2m-1$. If one of the edges of this complete graph has a prescribed color, then that color must be among the remaining available colors, so we can ensure that the prescribed edge receives its prescribed color. Therefore, the final edge-coloring is proper and is an extension of the precoloring.
\end{proof}

\begin{proposition} \label{odd}
    Let $G$ be a bipartite graph with maximum degree $\Delta(G)$ and $m \geq 1$ an integer. Suppose that the color of some edges of the Cartesian product $H = G \mathbin{\Box} K_{2m+1}$ is prescribed using at most $\Delta(G)+\chi'(K_{2m+1}) = \Delta(G) + 2m + 1$ colors such that the distance between any two precolored edges is at least $2$. Then, this precoloring can be extended to an edge-coloring of $H$ using $\Delta(G)+2m+1$ colors.
\end{proposition}

\begin{remark}
    \textup{The above proof of Theorem~\ref{main} also works similarly for Proposition~\ref{odd}. This establishes Conjecture \ref{Edwards} for this particular class of graphs. However, there is no evidence that this result is sharp: since $\chi'(H) = \Delta(G) + 2m$, a more exciting question is whether the above statement remains true for graphs of the form $H = G \mathbin{\Box} K_{2m+1}$ when the use of at most $\Delta(G) + 2m$ colors is allowed. }
\end{remark}


Applying Theorem \ref{main} for $m=1$ successively $d$ times, we obtain the following statement.

\begin{cor} \label{hyp-thm}
    Let $G$ be a bipartite graph with maximum degree $\Delta(G)$ and $m \geq 1$ an integer. Suppose that the color of some edges of the Cartesian product $G \mathbin{\Box} Q_m$ is prescribed using at most $\Delta(G)+m$ colors such that the distance between any two precolored edges is at least $2$. Then, this precoloring is extendable.
\end{cor}

\section{An analogous result for complete bipartite graphs}

Although establishing the case for $G \mathbin{\Box} K_2$ would be sufficient to prove Conjecture \ref{hypercube}, Theorem \ref{main} provides a broader generalization considering the Cartesian product $G \mathbin{\Box} K_{2m}$. However, since $K_2$ can also be viewed as the complete bipartite graph $K_{1,1}$, another natural direction for generalization would be to consider $G \mathbin{\Box} K_{n,n}$ or $G \mathbin{\Box} K_{n,m}$. This raises the question of whether the theorem remains valid under such an extension.


\begin{conjecture} \label{general}
    Let $G$ be a bipartite graph with maximum degree $\Delta(G)$. Suppose that the color of some edges of the Cartesian product $G \mathbin{\Box} K_{n,m}$ $(n \leq m)$ is prescribed using at most $\Delta(G)+m$ colors such that the distance between any two precolored edges is at least $2$. Then, this precoloring is extendable.
\end{conjecture}

\begin{remark}
    \textup{Theorem \ref{main} shows that Conjecture \ref{general} is true for the case $n=m=1$, and, since $K_{2,2}=K_2 \mathbin{\Box} K_2$, also in the case $n=m=2$.}
\end{remark}

In the following theorem we prove Conjecture \ref{general} for stars.

\begin{theorem} \label{star}
    Let $G$ be a bipartite graph with maximum degree $\Delta(G)$. Suppose that the color of some edges of the Cartesian product $G \mathbin{\Box} K_{1,m}$ is prescribed using at most $\Delta(G)+m$ colors such that the distance between any two precolored edges is at least $2$. Then, this precoloring is extendable.
\end{theorem}

\begin{proof} [Proof of Theorem \ref{star}]
    Note that $K_{1,m}$ is an induced subgraph of the hypercube $Q_m$, therefore, for any graph $G$, $G \mathbin{\Box} K_{1,m}$ is an induced subgraph of the Cartesian product $G \mathbin{\Box} Q_m$. Fix $G \mathbin{\Box} K_{1,m}$ as an induced subgraph of $G \mathbin{\Box} Q_m$. Then any distance-2 matching in $G \mathbin{\Box} K_{1,m}$ is also a distance-2 matching in $G \mathbin{\Box} Q_m$. 
    
    Consider the precoloring of the graph $G \mathbin{\Box} K_{1,m}$ and embed it into $G \mathbin{\Box} Q_m$. According to Corollary \ref{hyp-thm}, this precoloring is extendable and the restriction of this extension to $G \mathbin{\Box} K_{1,m}$ yields an extension of the precoloring to $G \mathbin{\Box} K_{1,m}$.
\end{proof}

It is quite easy to construct graphs $G$ and $H$ even, both regular and bipartite, such that not every precolored distance-2 matching in $G \mathbin{\Box} H$ is extendable (see: \cite{cassel-cart}). However, it is not possible that there exists a family of bipartite graphs such that if $H$ belongs to
this family, then for any bipartite graph $G$, there is a precolored distance-2 matching in $G \mathbin{\Box} H$ that is not extendable. For example, if we choose $G$ to be a star, Theorem \ref{star} shows that this cannot be true. If we allow ourselves to restrict the choice of $G$ and $H$, then it is possible to formulate statements of this form. The following is an example.


\begin{proposition} \label{negative}
    If $G$ and $H$ are bipartite graphs such that for both there exists a vertex of maximum degree whose neighborhood can be covered by an induced matching, then there exists a precolored distance-2 matching in $G \mathbin{\Box} H$ that is not extendable.
\end{proposition}

\begin{remark}
    \textup{For example, if both $G$ and $H$ are trees each with a vertex of maximum degree which is not adjacent to any leaf, then there exists a precolored distance-2 matching in $G \mathbin{\Box} H$ that is not extendable.}
\end{remark}

\begin{proof}[Proof of Proposition \ref{negative}]
    Let $a \in V(G)$ and $b \in V(H)$ be two vertices that satisfy the conditions of the statement. Furthermore, let $c_1d_1, \dots, c_id_i \in E(G)$ and $k_1m_1, \dots, k_jm_j \in E(H)$ be independent edge sets that cover the neighborhoods of $a$ and $b$, respectively, and such that the subgraphs induced by their endpoints are matchings. Precolor the edges $(c_1,b)(d_1,b), \dots, (c_i,b)(d_i,b), (a,k_1)(a,m_1) ,\dots, (a,k_j)$ $(a,m_j)$ in $G \mathbin{\Box} H$ with the same color. These edges form a distance-2 matching, according to the given conditions, and this precoloring is not extendable. In fact, $(a,b)$ is a maximum degree vertex in $G \mathbin{\Box} H$, so every color must appear on the edges adjacent to $(a,b)$. However, each of these edges is incident with a precolored edge, and thus none of them can receive the preassigned color.
\end{proof}

\section*{Statements and declarations}
The author declares that they have no conflict of interest.
Data sharing is not applicable to this article, as no datasets were generated or analyzed during the current study.

\textbf{Acknowledgement.} I am grateful to Ervin Győri for his useful observations and comments and for his support during the preparation of this paper. I also thank the reviewers for their thorough work. Their detailed and careful feedback not only improved the quality of the paper and made it more accessible, but also contributed to my own development.

\end{document}